\documentclass[A4, 11pt]{article}
\usepackage{amscd, amssymb, amsmath, amsthm}
\usepackage{latexsym}
\usepackage{upref}
\usepackage{epsfig}
\usepackage{mathrsfs}
\usepackage{float, floatflt}
\usepackage{indentfirst}
\usepackage{hyperref}
\hypersetup{
    colorlinks=true,
    linkcolor=blue,
    citecolor=blue,
    filecolor=blue,
    urlcolor=blue,
}
\usepackage{graphics,graphicx,color}
\usepackage{cite}
\usepackage{enumitem}
\usepackage{fancyhdr} 
\usepackage[utf8]{inputenc}         
\usepackage[T1]{fontenc}
\usepackage{mathpazo} 
\usepackage{avant} 

\AtBeginDocument{
  \DeclareSymbolFont{AMSb}{U}{msb}{m}{n}
  \DeclareSymbolFontAlphabet{\mathbb}{AMSb}}  

\usepackage{rotating} 
\usepackage{booktabs} 
\usepackage{pb-diagram,pb-xy}       
\usepackage{tikz}                   
\usetikzlibrary{matrix,arrows}

\theoremstyle{plain}
\newtheorem{thm}{Theorem}[section]

\newtheorem{prop}[thm]{Proposition}

\theoremstyle{definition}
\newtheorem{defn}[thm]{Definition}
\newtheorem{example}[thm]{Example}
\theoremstyle{remark}

\numberwithin{equation}{section}

\makeatletter
\def\th@plain{%
  \thm@notefont{}
  \itshape 
}
\def\th@definition{%
  \thm@notefont{}
  \normalfont 
} \makeatother

\setlist{font=\normalfont}

\DeclareMathAlphabet{\cols}{OMS}{cmsy}{m}{n} %

\newcommand{\Bp}[1]{\left(#1\right)}
\newcommand{\lbar}{\overline}
\newcommand{\GL}[1]{\mathrm{GL}({#1})}
\newcommand{\F}{\mathbb{F}}
\newcommand{\set}[1]{\{#1\}}
\newcommand{\C}{\mathbb{C}}
\newcommand{\R}{\mathbb{R}}
\newcommand{\D}{\mathbb{D}}
\newcommand{\cset}[2]{\set{{#1}\colon{#2}}}
\newcommand{\abs}[1]{|#1|}
\newcommand{\gyr}[2]{{\mathrm{gyr}[{#1}]}{#2}}

\newcommand{\id}[1]{\mathrm{id}_{#1}}

\newcommand{\lsum}[2]{\displaystyle\sum_{#1}^{#2}}
\newcommand{\gen}[1]{\langle#1\rangle}

\newcommand{\gyrL}[1]{L^\mathrm{gyr}(#1)}
\newcommand{\igyr}[2]{{\mathrm{gyr^{-1}}[{#1}]}{#2}}
\newcommand{\Cset}[2]{\left\{{#1}\colon{#2}\right\}}
\newcommand{\qt}[1]{``#1''}
\newcommand{\red}[1]{\begin{color}{black}#1\end{color}}

\DeclareMathOperator{\im}{im}
\DeclareMathOperator{\aut}{Aut}
\DeclareMathOperator{\cha}{char}
\DeclareMathOperator{\Fix}{Fix}

\newcommand{\vphi}{\varphi}

\begin{document}
\title{\bf Extension of Maschke's theorem\footnote{This research is part of the project {\it Gyrogroups and their linear representations}, funded by the Institute for Promotion of Teaching Science and Technology (IPST), Thailand.
}}
\author{Teerapong Suksumran\,\href{https://orcid.org/0000-0002-1239-5586}{\includegraphics[scale=1]{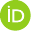}}\\
Center of Excellence in Mathematics and Applied Mathematics\\
Department of Mathematics\\
Faculty of Science, Chiang Mai University\\
Chiang Mai 50200, Thailand\\
{\tt teerapong.suksumran@cmu.ac.th}}
\date{}
\maketitle


\begin{abstract}
In the present article, we examine linear representations of finite gyrogroups, following their group-counterparts. In particular, we prove the celebrated theorem of Maschke for  gyrogroups, along with its converse. This suggests studying the left regular action of a gyro-group $(G, \oplus)$ on the function space
$$
L^{\mathrm{gyr}}(G) = \{f\in L(G)\colon \forall a, x, y, z\in G, f(a\oplus\gyr{x, y}{z}) = f(a\oplus
z)\}
$$
in a natural way, where $L(G)$ is the space of all functions from $G$ into a field.
\end{abstract}
\thispagestyle{empty}

\section{Introduction}
Maschke's theorem for groups states that if $G$ is a finite group and if $\F$ is a field whose characteristic does not divide the order of $G$ (\mbox{including} fields of characteristic zero), then the group algebra $\F[G]$ of $G$ over $\F$ is \mbox{semisimple} or, equivalently, every submodule of an $\F[G]$-module is a direct summand. In terms of linear representations, Maschke's theorem states, under the same hypotheses, that if $\vphi\colon G\to \GL{V}$, where $V$ is a vector space over $\F$, defines a linear representation, then every invariant subspace of $V$ has an invariant direct sum complement. The famous \mbox{theorem} of Maschke is still far from being exhausted, as one can see for \mbox{example} in \cite{DP1983IEM, WZLZ2011MTS, MD1981MTS, MGMS1996MVM} and, of course, in this  article. Inspired by the latter form of Maschke's theorem, we generalize this to {\it gyrogroups}, a suitable extension of groups.  The converse to Maschke's theorem treats the case where the \mbox{characteristic} of $\F$ divides the order of $G$ and motivates the theory of \mbox{\it modular} representations.

In \cite{TS2015GAG} we introduce the notion of gyrogroup actions, which amounts to that of {\it permutation} representations of a gyrogroup $G$ on a nonempty set $X$, as a generalization of group actions. This results in gyrogroup \mbox{versions} of three well-known theorems in group theory: the orbit-stabilizer theorem, the orbit decomposition theorem, and the Burnside lemma (or the Cauchy--Frobenius lemma). When $X$ admits the linear structure, the method of \mbox{gyrogroup} action yields {\it linear} representations: to represent elements of $G$ by linear transformations on $X$, as discussed in some detail in \cite{TSKW2017MFE}. The notable result in \cite{TSKW2017MFE} is Schur's lemma for \mbox{gyrogroups}, with applications to the open unit disk of the complex plane. The study of permutation and linear representations of gyrogroups leads to a better understanding of \mbox{gyrogroup} structures from the algebraic viewpoint. For this reason, we continue to examine linear representations of (finite) gyrogroups, \mbox{especially} Maschke's theorem and its converse.

\section{Preliminaries}
\subsection{Gyrogroups and their basic properties}
Gyrogroups first arose as an algebraic structure that underlies the space of relativistically admissible velocities in $\R^3$ in special relativity \cite{AU2007EVA}. They provide a powerful tool for studying analytic hyperbolic geometry; see, for instance, \cite{AU2008AHG} and references therein. \red{For an introduction to the formation of a gyrogroup, we refer the reader to Preface and Chapter 1 of \cite{AU2008AHG}.} The formal definition of a gyrogroup is as follows. Let $G$ be a nonempty set equipped with a binary operation $\oplus$ on $G$ and let $\aut{G}$ be the group of automorphisms of $(G, \oplus)$.

\begin{defn}[Gyrogroups]\label{def: gyrogroup}
A system $(G,\oplus)$ is called a {\it gyrogroup} if its binary
operation satisfies the following axioms.\newpage
\begin{enumerate}[label=(\alph*),leftmargin=1.5\parindent]
    \item[(G1)] There is an element $e\in G$ such that $e\oplus a =
    a$ for all $a\in G$.
    \item[(G2)] For each $a\in G$, there is an element $b\in G$ such that
$b\oplus a = e$.
    \item[(G3)] For all $a$, $b\in G$, there is an automorphism
$\gyr{a,b}{}\in\aut{G}$ such that
    \begin{equation}\tag{left gyroassociative law} a\oplus (b\oplus c) = (a\oplus b)\oplus\gyr{a,
    b}{c}\end{equation}
    for all $c\in G$.
    \item[(G4)] For all $a$, $b\in G$, $\gyr{a\oplus b, b}{} = \gyr{a,
    b}{}$.\hfill(left loop property)
\end{enumerate}
\end{defn}

\red{
The axioms in Definition \ref{def: gyrogroup} imply the right counterparts. In fact, they imply that any gyrogroup has a unique two-sided identity, denoted by $e$ (cf. Theorems 2.10(5) and 2.10(6) of \cite{AU2008AHG}), and that any element $a$ in a \mbox{gyrogroup} has a unique two-sided inverse, denoted by $\ominus a$ (cf. Theorems 2.10(7) and 2.10(8) of \cite{AU2008AHG}). Furthermore, any gyrogroup satisfies the {\it right \mbox{gyroassociative} law}, 
\begin{equation}
(a\oplus b)\oplus c = a\oplus(b\oplus\gyr{b, a}{c}),
\end{equation}
as well as the {\it right loop property}, 
\begin{equation}
\gyr{a, b\oplus a}{} = \gyr{a, b}{}.
\end{equation}
Another useful identity in gyrogroups is the {\it inversively symmetric property},
\begin{equation}
\igyr{a, b}{} = \gyr{b, a}{},
\end{equation}
where $\igyr{a, b}{}$ denotes the inverse function of $\gyr{a, b}{}$; see Theorem 2.34 of \cite{AU2008AHG} for its proof.
}

It turns out that gyrogroups share remarkable analogies with groups; see, for instance, \cite{TS2015GAG, TS2016TAG, TSKW2017MFE, TSKW2014LTG}. The most interesting part of a gyrogroup is the \mbox{automorphism} $\gyr{a, b}{}$ \red{mentioned} in Definition \ref{def: gyrogroup}, called the {\it gyroautomorphism} generated by $a$ and $b$. This is because the gyroautomorphisms encode all the information about the {\it gyroassociative law}, an analogue of the \mbox{associative} law in group theory. Further, \red{the loop property} plays a \mbox{fundamental} role in gyrogroup theory, as it forces gyrogroups to have rigid structures. Note that any gyrogroup with {\it trivial} gyroautomorphisms forms a group and, conversely, any group forms a gyrogroup by defining the gyroautomorphisms to be the identity automorphism. The \red{study of basic properties of gyrogroups can be found, for example, in \cite{AU2008AHG, TS2016TAG}}.

\red{
\begin{example}
A prime example of a gyrogroup is the {\it M\"{o}bius gyrogroup}\cite{AU2008FMG}, consisting of the unit disk $\D=\cset{z\in\C}{\abs{z}<1}$ with M\"{o}bius addition \mbox{defined} by 
\begin{equation}
a\oplus_M b = \dfrac{a+b}{1+\lbar{a}b},\qquad a, b\in\D.
\end{equation}
Given $a, b\in\D$, the gyroautomorphism of $(\D, \oplus_M)$ generated by $a$ and $b$ is defined by
\begin{equation}\label{eqn: Mobius gyroautomorphism}
\gyr{a, b}{z} = \dfrac{1+a\lbar{b}}{1+\lbar{a}b}z
\end{equation}
for all $z\in\D$. Since $\dfrac{1+a\lbar{b}}{1+\lbar{a}b}$ is a unimodular complex number, Equation \eqref{eqn: Mobius gyroautomorphism} represents a rotation of the unit disk. This justifies the use of the prefix \qt{gyro}. We remark that $(\D, \oplus_M)$ does not form a group because M\"{o}bius addition fails to satisfy the associative law.
\end{example}

Examples of {\it nondegenerate} finite gyrogroups (that is, finite gyrogroups that are not groups) do exist; see, for instance, Example \ref{ex: G8} below. According to Theorem 6.4 of \cite{HK2002TKL}, gyrogroups and left Bol loops in which every left inner mapping  is an automorphism are equivalent algebraic structures (for the relevant definitions, see Chapter 6 of \cite{HK2002TKL}). This combined with Burn's results (Theorems 4 and 5 of \cite{RB1978FBL}) implies that finite gyrogroups of orders $2p$ and $p^2$, where $p$ is a prime, are groups. Further, any gyrogroup of prime order is a group (cf. Theorem 6.2 of \cite{TSKW2014LTG}). Therefore, the smallest possible order of a nondegenerate finite gyrogroup is $8$; one example is presented below. For a construction of a gyrogroup from a nilpotent group of class $3$, we refer the reader to Corollary 3.8 of \cite{TFAU2001GDG}.

\begin{example}\label{ex: G8}
The gyrogroup $G_8 = \set{0, 1, \ldots, 7}$ is exhibited in Example 1 on p. 404 of \cite{TS2016TAG}. Its Cayley table and gyration table are presented in Tables \ref{tab: gyrogroup G8} and \ref{tab: gyration table of G8}, respectively. There are only two gyroautomorphisms of $G_8$; one is the identity automorphism denoted by $I$ and the other is the automorphism $\tau$ given by 
\begin{eqnarray}\label{eqn: gyroautomorphism of G8}
\begin{split}
0 \mapsto 0 &{\hskip1cm}& 4 \mapsto 6\\
1 \mapsto 1 &{\hskip1cm}& 5 \mapsto 7\\
2 \mapsto 2 &{\hskip1cm}& 6 \mapsto 4\\
3 \mapsto 3 &{\hskip1cm}& 7 \mapsto 5
\end{split}
\end{eqnarray}
Since $\tau$ is nontrivial, $G_8$ does not form a group.

\begin{table}[ht]
\centering\red{
\begin{tabular}{|c|cccccccc|}
\hline
$\oplus$ & 0 & 1 & 2 & 3 & 4 & 5 & 6 & 7 \\ \hline
0 & 0 & 1 & 2 & 3 & 4 & 5 & 6 & 7 \\
1 & 1 & 3 & 0 & 2 & 7 & 4 & 5 & 6 \\
2 & 2 & 0 & 3 & 1 & 5 & 6 & 7 & 4 \\
3 & 3 & 2 & 1 & 0 & 6 & 7 & 4 & 5 \\
4 & 4 & 5 & 7 & 6 & 3 & 2 & 0 & 1 \\
5 & 5 & 6 & 4 & 7 & 2 & 0 & 1 & 3 \\
6 & 6 & 7 & 5 & 4 & 0 & 1 & 3 & 2 \\
7 & 7 & 4 & 6 & 5 & 1 & 3 & 2 & 0 \\ \hline
\end{tabular}}
\caption{Cayley table for the gyrogroup
$G_8$.}\label{tab: gyrogroup G8}
\end{table}

\begin{table}[ht]
\centering\red{
\begin{tabular}{|c|cccccccc|}
\hline
$\mathrm{gyr}$ & 0 & 1 & 2 & 3 & 4 & 5 & 6 & 7 \\ \hline
0 & $I$ & $I$ & $I$ & $I$ & $I$ & $I$ & $I$ & $I$ \\
1 & $I$ & $I$ & $I$ & $I$ & $\tau$ & $\tau$ & $\tau$ & $\tau$ \\
2 & $I$ & $I$ & $I$ & $I$ & $\tau$ & $\tau$ & $\tau$ & $\tau$ \\
3 & $I$ & $I$ & $I$ & $I$ & $I$ & $I$ & $I$ & $I$ \\
4 & $I$ & $\tau$ & $\tau$ & $I$ & $I$ & $\tau$ & $I$ & $\tau$ \\
5 & $I$ & $\tau$ & $\tau$ & $I$ & $\tau$ & $I$ & $\tau$ & $I$ \\
6 & $I$ & $\tau$ & $\tau$ & $I$ & $I$ & $\tau$ & $I$ & $\tau$ \\
7 & $I$ & $\tau$ & $\tau$ & $I$ & $\tau$ & $I$ & $\tau$ & $I$ \\ \hline
\end{tabular}}
\caption{Gyration table for $G_8$. The automorphism $\tau$ is given by \eqref{eqn: gyroautomorphism of G8}.}\label{tab: gyration table of G8}
\end{table}
\end{example}
}

\subsection{Actions and linear representations of gyrogroups}

The notion of linear representations of a (finite or infinite) gyrogroup was formulated in \cite{TSKW2017MFE} in order to study the M\"{o}bius functional equation,
$$
L\Bp{a\oplus_M b} = L(a)L(b),
$$
where $L$ is a complex-valued function defined on $\D$. Next, we summarize basic knowledge of linear representations of a gyrogroup mentioned in \cite{TSKW2017MFE} for reference.

Let $(G, \oplus)$ be a gyrogroup, let $V$ be a vector space over an arbitrary field $\F$, and let $\GL{V}$ be the general linear group of $V$. A {\it linear representation} $\vphi$ of $G$ on $V$, denoted by $(V, \vphi)$, is a gyrogroup homomorphism from $G$ into $\GL{V}$; that is, $\vphi\colon G\to \GL{V}$ satisfies
$$
\vphi(a\oplus b) = \vphi(a)\circ\vphi(b)
$$
for all $a, b\in G$. A {\it linear action} of $G$ on $V$ is a map from $G\times V$ to $V$, written $(a, v)\mapsto a\cdot v$, such that
\begin{enumerate}
    \item\label{item: 0.v} $e\cdot v = v$ for all $v\in V$, $e$ being the identity of $G$,
    \item\label{item: a.(b.v)} $a\cdot(b\cdot v) = (a\oplus b)\cdot v$ for all $a$, $b\in
    G$, $v\in V$, and
    \item\label{item: sigma a is linear} for each $a\in G$, the map defined by
    \begin{equation}\label{eqn: induced linear map}
    v\mapsto a\cdot v,\qquad v\in V,
    \end{equation}
    is a linear transformation on $V$.
\end{enumerate}
In this case, $G$ is said to act {\it linearly} on $V$. According to Lemma 3.3 and Theorem 3.4 of \cite{TSKW2017MFE}, if $G$ acts linearly on $V$, then the map $\vphi(a)\colon v\mapsto a\cdot v$, $v\in V$, defines a linear automorphism of $V$ for all $a\in G$, and the map $a\mapsto \vphi(a)$, $a\in G$, defines a linear representation of $G$ on $V$. Conversely, if $\vphi$ is a linear representation of $G$ on $V$, then $G$ acts linearly on $V$ by defining $a\cdot v = \vphi(a)(v)$ for all $a\in G, v\in V$ \cite[Theorem 3.5]{TSKW2017MFE}. Recall that the {\it degree} or {\it dimension} of a linear representation $(V, \vphi)$ is defined as the dimension of $V$.

Let $G$ act linearly on $V$. A subspace $W$ of $V$ is said to be {\it invariant} if $\vphi(a)(W)\subseteq W$ for all $a\in G$ or, equivalently, $a\cdot w\in W$ for all $a\in G, w\in W$. Let $(V, \vphi)$ and $(W, \psi)$ be linear representations of $G$. A  linear transformation $\Phi\colon V\to W$ is called  an {\it intertwining map} if
\begin{equation}
\Phi\circ\vphi(a) = \psi(a)\circ\Phi
\end{equation}
for all $a\in G$; that is, if
\begin{equation}
\Phi(a\cdot v) = a\cdot\Phi(v)
\end{equation}
for all $a\in G$, $v\in V$. A bijective intertwining map is called an {\it equivalence}.   If there exists an equivalence from $V$ to $W$, we say that $(V, \vphi)$ and $(W, \psi)$ are {\it equivalent}. Of course, equivalent representations are algebraically identical and carry the same algebraic information.

\begin{defn}[Definition 3.9, \cite{TSKW2017MFE}]\label{def: irreducible representation}
A linear representation of $G$ on $V$ is {\it irreducible} if  the only invariant subspaces of $V$ are $\set{0}$ and $V$ itself.
\end{defn}

\begin{defn}[Decomposable representations]\label{def: decomposable representation}
A linear representation $(V, \vphi)$ of a gyrogroup is {\it decomposable} if there are nontrivial proper \mbox{invariant} subspaces $U$ and $W$ of $V$ such that $V = U\oplus W$. A linear representation is {\it indecomposable} if it \red{is} not decomposable.
\end{defn}

It follows directly from Definitions \ref{def: irreducible representation} and \ref{def: decomposable representation} that every irreducible linear representation of a gyrogroup is indecomposable. Irreducible linear representations of a gyrogroup are studied in some detail in \cite{TSKW2017MFE}.

\begin{defn}[Completely reducible representations]\label{def: completely reducible}
A linear representation $(V, \vphi)$ of a gyrogroup is {\it completely reducible} if there are invariant subspaces $V_1, V_2,\ldots, V_n$ of $V$ such that $V = V_1\oplus V_2\oplus\cdots\oplus V_n$ and the only invariant subspaces of $V_i$ are $\set{0}$ and $V_i$ itself for all $i = 1,2,\ldots, n$; that is, the restriction of $\vphi$ on $V_i$ is irreducible for all $i = 1,2,\ldots, n$.
\end{defn}

The following theorem indicates that the property of being irreducible (respectively, decomposable, indecomposable, and completely reducible) is an algebraic invariant of linear representations of gyrogroups.

\begin{thm}\label{thm: invarint of being irreducible, decomposable or completely reducible}
Let $(V, \vphi)$ and $(W, \psi)$ be equivalent linear representations of a gyrogroup. 
\begin{enumerate}
\item\label{item: irreducible representation, equivalent} If $\vphi$ is irreducible, then so is $\psi$.
\item\label{item: decomposable representation, equivalent} If $\vphi$ is decomposable, then so is $\psi$.
\item\label{item: completely reducible representation, equivalent} If $\vphi$ is completely reducible, then so is $\psi$.
\end{enumerate}
\end{thm}
\begin{proof}
The theorem follows from the fact that \hyperlink{x}{(i)} if $U$ is a nontrivial proper invariant subspace of $V$ and if $\Phi\colon V\to W$ is an equivalence, then $\Phi(U)$ is a nontrivial proper invariant subspace of $W$ and \hyperlink{x}{(ii)} if $\Phi\colon V\to W$ is a linear isomorphism and $V = V_1\oplus V_2\oplus\cdots\oplus V_n$ represents a direct sum decomposition, then so does $W = \Phi(V_1)\oplus\Phi(V_2)\oplus\cdots\oplus\Phi(V_n)$.
\end{proof}

\section{A gyrogroup version of Maschke's theorem}
It is well known that if $U$ is a subspace of a vector space $V$, then $U$ has a direct sum complement, a subspace $W$ of $V$ such that $V = U\oplus W$. When a linear representation $(V, \vphi)$ of a finite gyrogroup $G$ is given and $U$ happens to be an invariant subspace of $V$, $W$ can be \mbox{chosen} invariant. This remarkable property relies on the fact that the order of $G$ is finite and invertible in the base \red{field} and hence the traditional \qt{averaging trick} becomes available. The following theorem, which is of interest in its own right, is a preparation of {\it Maschke's theorem} for gyrogroups. Its proof is quite elaborate, as we intend to show how gyrogroup theory comes into play in representation theory.

\begin{thm}\label{thm: intertwining projection}
Let $(V, \vphi)$ be a linear representation of a finite gyrogroup $G$ over $\F$ and let $U$ be an invariant subspace of $V$. If
$\cha{\F} = 0$ or $\cha{\F}$ does not divide $\abs{G}$, then there exists a projection $\pi$ of $V$ onto $U$ that is an intertwining map. In other words, $\pi$ satisfies the following properties:
\begin{enumerate}
    \item $\pi$ is linear;
    \item $\pi(u) = u$ for all $u\in U$;
    \item $\pi^2 = \pi$;
    \item $\pi(a\cdot v) = a\cdot\pi(v)$ for all $a\in G, v\in V$.
\end{enumerate}
\end{thm}
\begin{proof} 
By Theorem 1.4 of \cite{SR2008ALA}, there exists a subspace $W_0$ of $V$
such that \mbox{$V = U\oplus W_0$}. Let $\pi_0$ be the projection of $V$
onto $U$ associated to this direct sum decomposition. For each $a\in
G$, define a map $\phi(a)$ by
$$
\phi(a) = \vphi(a)\circ\pi_0\circ\vphi(\ominus a).
$$
Being the composition of linear transformations on $V$, $\phi(a)$ is
linear for all $a\in G$. Further, since $U$ is invariant,
$\phi(a)(V)\subseteq U$ for all $a\in G$. This implies
$$\phi(a)(u) = \vphi(a)(\pi_0(\vphi(\ominus a)(u))) = \vphi(a)(\vphi(\ominus a)(u)) = \vphi(e)(u) = u$$
for all $a\in G, u\in U$.

\par By assumption, $n := \abs{G}1 =
\underbrace{1+1+\cdots+1}_{\abs{G}\textrm{ copies}}$ is nonzero in
$\F$. Thus, $n$ is invertible in $\F$. Define
\begin{equation}
\pi = \frac{1}{n}\lsum{a\in G}{}\phi(a).
\end{equation}
Then $\pi$ is a linear transformation from $V$ into $U$ such that
$\pi(u) = u$ for all $u\in U$ and hence $\pi(\pi(v)) = \pi(v)$ for all $v\in V$. It is easy to see that $\pi$ is surjective. Next, we prove that $\pi$ is an intertwining map.
Let $b\in G$ and let $v\in V$. Using the fact that $\vphi$ is a gyrogroup homomorphism and $\vphi(\ominus a) = \vphi(a)^{-1}$ for all $a\in G$ \red{(cf. Proposition 32(2) of \cite{TS2016TAG})}, we obtain

\begin{align}
\begin{split}\label{eqn: in proof of Maschke theorem}
\pi(b\cdot v) &= \frac{1}{n}\lsum{a\in G}{}
(\vphi(a)\circ\pi_0\circ\vphi(\ominus a))(\vphi(b)(v))\\
{} &= \frac{1}{n}\lsum{a\in G}{}(\vphi(b)\circ\vphi(\ominus a\oplus
b)^{-1}\circ\pi_0\circ\vphi(\ominus a\oplus b))(v)\\
{} &= \frac{1}{n}\lsum{a\in G}{}(\vphi(b)\circ\vphi(\ominus(\ominus
a\oplus b))\circ\pi_0\circ\vphi(\ominus a\oplus b))(v)\\
{} &= \vphi(b)\Bp{\frac{1}{n}\lsum{a\in G}{}(\vphi(\ominus(\ominus
a\oplus b))\circ\pi_0\circ\vphi(\ominus a\oplus b))(v)}.
\end{split}
\end{align}
Let $\iota$ and $R_b$ be the maps defined by $\iota(a) = \ominus a$ and $R_b(a) = a\oplus b$ for all $a\in G$. Then $\iota$ is a permutation of $G$ and by Theorem 2.22 of \cite{AU2008AHG}, $R_b$ is a permutation
of $G$. So, the composite $\theta := \iota\circ R_b\circ \iota$ is a
permutation of $G$. Hence, if $a$ runs over all elements of
$G$, then so does $c = \theta(a) = \ominus(\ominus a\oplus b)$.
Therefore, \eqref{eqn: in proof of Maschke theorem} becomes
\begin{align*}
\pi(b\cdot v) &= \vphi(b)\Bp{\frac{1}{n}\lsum{c\in G}{}(\vphi(c)\circ\pi_0\circ\vphi(\ominus c)(v)}\\ 
{} &= \vphi(b)\Bp{\dfrac{1}{n}\lsum{c\in G}{}\phi(c)(v)}\\ 
{} &= \vphi(b)(\pi(v))\\ 
{} &= b\cdot\pi(v).\tag*{\qed}
\end{align*}
\let\qed\relax
\end{proof}

We are now in a position to prove Maschke's theorem for gyrogroups. The proof is analogous to the case of groups, as gyrogroups \red{share common properties with} groups.

\begin{thm}[Maschke's theorem]
Let $(V, \vphi)$ be a linear representation of a finite gyrogroup $G$ over $\F$. If $\cha{\F} = 0$ or $\cha{\F}$ does not divide
$\abs{G}$, then for any invariant subspace $U$ of $V$, there exists an invariant subspace $W$ of $V$ such that $V = U\oplus W$.
\end{thm}
\begin{proof}
Suppose that $U$ is an invariant subspace of $V$ and let $\pi$ be as in \mbox{Theorem} \ref{thm: intertwining projection}. Set
$W = \ker{\pi}$. We claim that $V = U\oplus W$. In fact, for each
\mbox{$v\in V$}, $v = \pi(v)+(v-\pi(v))$. Note that $$\pi(v-\pi(v)) =
\pi(v)-\pi(\pi(v)) = 0.$$ Hence, $v-\pi(v)\in \ker{\pi} = W$. This
proves $V = U+W$. If $u\in U\cap W$, then $u = \pi(u) = 0$. Thus,
$U\cap W = \set{0}$ and so the sum $V = U+W$ is direct. Since $\pi$
is an intertwining map of $V$, it follows that $W = \ker{\pi}$ is an invariant
subspace of $V$ by Lemma 3.10 of \cite{TSKW2017MFE}.
\end{proof}

It is clear from Definition \ref{def: completely reducible} that every irreducible linear representation of a gyrogroup is completely reducible. Next, we prove that every {\it finite-dimensional} linear representation of a {\it finite} gyrogroup over a particular field is completely reducible, as a consequence of Maschke's theorem.

\begin{thm}\label{thm: complete reducibility of finite representation of finite gyrogroup}
Let $G$ be a finite gyrogroup. If $\cha{\F} = 0$ or $\cha{\F}$ does not divide $\abs{G}$, then every finite-dimensional linear representation of $G$ over $\F$ is completely reducible.
\end{thm}
\begin{proof}
We proceed by induction on the dimension of a linear representation. Let $(V,  \vphi)$ be a finite-dimensional linear representation of $G$ over $\F$. If $\dim{(V, \vphi)} = 1$, then $V$ is one dimensional. Hence, $V$ contains no nontrivial proper subspaces and so $\vphi$ is irreducible. As noted earlier, $\vphi$ is completely reducible. Suppose that $\dim{(V, \vphi)} = n$. If $\vphi$ is irreducible, then $\vphi$ is completely reducible. Therefore, we may assume that $\vphi$ is not irreducible and so a nontrivial proper invariant subspace $U$ of $V$ exists. By Maschke's theorem, $V = U\oplus W$ for some invariant subspace $W$ of $V$. The inductive hypothesis implies that $U = U_1\oplus U_2\oplus\cdots\oplus U_m$ and $W = W_1\oplus W_2\oplus\cdots\oplus W_k$, where $U_i$ is an invariant subspace of $U$ for all $i$, $W_j$ is an invariant subspace of $W$ for all $j$, and  $U_i$ and $W_j$ have no non-trivial proper invariant subspaces  for all $i = 1,2,\ldots, m, j = 1,2,\ldots, k$. Since $V = U_1\oplus U_2\oplus\cdots\oplus U_m\oplus W_1\oplus W_2\oplus\cdots\oplus W_k$, it follows that $\vphi$ is completely reducible. This completes the induction.
\end{proof}

Theorem \ref{thm: complete reducibility of finite representation of finite gyrogroup} shows that the study of finite-dimensional linear representations of a finite gyrogroup $G$ over a field $\F$, where $\cha{\F} = 0$ or $\cha{\F}$ does not divide $\abs{G}$, reduces to the study of the irreducible ones. Furthermore, if $G$ happens to be a {\it gyrocommutative} gyrogroup and $\F$ is \mbox{algebraically} closed, then every irreducible linear representation of $G$ is one dimensional. This is a consequence of Schur's lemma for gyrogroups, see Section 3.2 of \cite{TSKW2017MFE} for more details. 

It is well known in the literature that the converse to Maschke's \mbox{theorem} for groups also holds: if the characteristic of $\F$ divides the order of a finite group $G$, then there \mbox{exists} an invariant subspace induced by the left \mbox{regular} representation of $G$ that has no an \mbox{invariant} direct sum  complement. This leads to the study of the left regular representation of a gyrogroup in Section \ref{sec: left regular representation} and eventually to part of the converse to Maschke's theorem for \mbox{gyrogroups} in Section \ref{sec: converse of Maschke's theorem}.

\section{The left regular representation}\label{sec: left regular representation}
In this section, we provide an extremely important example of a linear representation of an arbitrary gyrogroup, namely the {\it left regular representation}. This representation will play a crucial role in the study of the converse to Maschke's theorem, as shown in Section \ref{sec: converse of Maschke's theorem}.

\par Let $G$ be a gyrogroup and let $\F$ be a field. Denote by $L(G)$ the space of all functions from $G$ into $\F$.
That is,
\begin{equation}
L(G) = \cset{f}{f \textrm{ is a function from $G$ into $\F$}}.
\end{equation}
Recall that $L(G)$ is a vector space over $\F$ whose vector addition and scalar multiplication are defined pointwise. Note that the zero function $o\colon x\mapsto 0$, $x\in G$, is the zero vector in $L(G)$ and that $-f\colon x\mapsto -f(x)$, $x\in G$, is the inverse of $f$ with respect to vector addition. If $G$ is finite, then $L(G)$ is finite dimensional, as shown in the following theorem. 

\begin{thm}
Let $G$ be a gyrogroup. For each $a\in G$, define a map $\delta_a$ on $G$ by
\begin{equation}
\delta_a(x) = 
\begin{cases}
1 & \textrm{if }x = a\\
0 & \textrm{otherwise.}
\end{cases}
\end{equation}
If $G$ is finite, then $B = \cset{\delta_a}{a\in G}$ forms a basis for $L(G)$ so that $L(G)$ is of finite dimension $\abs{G}$.
\end{thm}
\begin{proof}
The proof of the theorem is immediate.
\end{proof}

Unlike the situation of groups, $G$ does not act linearly on $L(G)$ for the associative law fails to hold in gyrogroups. This leads us to seek to find subspaces of $L(G)$ on which $G$ acts linearly. In fact, we define
\begin{equation}\label{eqn: definition of Lgyr(G)}
\gyrL{G}= \cset{f\in L(G)}{\forall a, x, y, z\in G, f(a\oplus\gyr{x, y}{z}) = f(a\oplus
z)}.
\end{equation}
It is straightforward to check that $\gyrL{G}$ forms a subspace of $L(G)$ and hence is a vector space over $\F$.

\begin{prop}\label{prop: relation between Lgyr(G) and L(G)}
Let $G$ be a gyrogroup. Then $\gyrL{G} = L(G)$ if and only if the
gyroautomorphisms of $G$ are trivial, that is, $G$ is a group.
\end{prop}
\begin{proof}
The converse statement is obvious. Suppose that $\gyrL{G} = L(G)$. Suppose to the contrary that $\gyr{x, y}{}\ne\id{G}$ for some $x, y\in G$. Then $\gyr{x, y}{z}\ne z$ for some $z\in G$. Set $a = \gyr{x, y}{z}$. By assumption, $\delta_a$ is in $\gyrL{G}$, which implies 
$1 = \delta_a(a) = \delta_a(\gyr{x, y}{z}) = \delta_a(z)$ and hence $z = a = \gyr{x, y}{z}$, a contradiction. Thus,  all the gyroautomorphisms of $G$ are trivial.
\end{proof}

\par Let $a\in G$. Recall that the {\it left gyrotranslation} by $a$,
$L_a$, is a permutation of $G$ defined by $L_a(x) = a\oplus x$ for
all $x\in G$ \cite[Theorem 18]{TS2016TAG}. The main result of this section is presented in the following theorem, demonstrating that $G$ acts linearly on $\gyrL{G}$ in a natural way.

\begin{thm}
Let $G$ be a gyrogroup. Then $G$ acts linearly on $\gyrL{G}$ by
\begin{equation}\label{eqn: left regular action on Lgyr(G)}
a\cdot f = f\circ L_{\ominus a}
\end{equation}
for all $a\in G$, $f\in\gyrL{G}$.
\end{thm}
\begin{proof}
Let $a\in G$ and let $f\in\gyrL{G}$. First, we prove that $a\cdot
f\in\gyrL{G}$. In fact, we have
\begin{eqnarray*}
(a\cdot f)(b\oplus\gyr{x, y}{z}) &=& (f\circ L_{\ominus
a})(b\oplus\gyr{x, y}{z})\\
{} &=& f(\ominus a\oplus(b\oplus\gyr{x, y}{z}))\\
{} &=& f((\ominus a\oplus b)\oplus\gyr{\ominus a, b}{(\gyr{x,
y}{z})})\\
{} &=& f((\ominus a\oplus b)\oplus\gyr{x, y}{z})\\
{} &=& f((\ominus a\oplus b)\oplus z)\\
{} &=& f((\ominus a\oplus b)\oplus\gyr{\ominus a, b}{z})\\
{} &=& f(\ominus a\oplus(b\oplus z))\\
{} &=& (f\circ L_{\ominus a})(b\oplus z)\\
{} &=& (a\cdot f)(b\oplus z)
\end{eqnarray*}
for all $b, x, y, z\in G$. Thus, $a\cdot f\in\gyrL{G}$.

\par Since $L_{\ominus e} = \id{G}$, we obtain $e\cdot f = f$ for all
$f\in\gyrL{G}$. Let $a, b\in G$ and let $f\in\gyrL{G}$. By the
defining property of $\gyrL{G}$, $f(\gyr{x, y}{z}) = f(z)$ for all
$x, y, z\in G$. Let $x\in G$. We compute
\begin{eqnarray*}
(a\cdot (b\cdot f))(x) &=& (b\cdot f)(\ominus a\oplus x)\\
{} &=& f(\ominus b\oplus(\ominus a\oplus x))\\
{} &\stackrel{\hypertarget{(a)}{\rm (a)}}{=}& f((\ominus b\ominus a)\oplus\gyr{\ominus b, \ominus a}{x})\\
{} &\stackrel{\hypertarget{(b)}{\rm (b)}}{=}& f((\ominus b\ominus a)\oplus x)\\
{} &\stackrel{\hypertarget{(c)}{\rm (c)}}{=}& f((\ominus b\ominus a)\oplus \gyr{b, a}{x})\\
{} &\stackrel{\hypertarget{(d)}{\rm (d)}}{=}& f(\gyr{a, b}{((\ominus b\ominus a)\oplus\gyr{b, a}{x}}))\\
{} &\stackrel{\hypertarget{(e)}{\rm (e)}}{=}& f(\gyr{a, b}{(\ominus b\ominus a)}\oplus x)\\
{} &\stackrel{\hypertarget{(f)}{\rm (f)}}{=}& f(\ominus (a\oplus b)\oplus x)\\
{} &=& (f\circ L_{\ominus(a\oplus b)})(x)\\
{} &=& ((a\oplus b)\cdot f)(x).
\end{eqnarray*}
We obtain \hyperlink{(a)}{(a)} from the left gyroassociative law; \hyperlink{(b)}{(b)}, \hyperlink{(c)}{(c)} and \hyperlink{(d)}{(d)} from \eqref{eqn: definition of Lgyr(G)}; \hyperlink{(e)}{(e)} from the fact that $\gyr{a, b}{}$
preserves $\oplus$ and $\gyr{b, a}{} = \igyr{a, b}{}$; and \hyperlink{(f)}{(f)} from the identity $\ominus (a\oplus b) = \gyr{a, b}{(\ominus b\ominus a)}$, \red{see Theorem 2.32 of \cite{AU2008AHG}}. Since $x$ is arbitrary, it follows that $a\cdot (b\cdot f) = (a\oplus b)\cdot f$.

\par For each $a\in G$, let $\lambda(a)$ be the map defined by  $\lambda(a)(f) = a\cdot f$ for all $f\in \gyrL{G}$. An easy verification shows that $\lambda(a)$ is a linear transformation on $\gyrL{G}$ for all $a\in G$. 
\end{proof}

\begin{defn}[The left regular representation]
The linear action given by \eqref{eqn: left regular action on Lgyr(G)} is called the {\it left regular action} of $G$ on $\gyrL{G}$ and its corresponding linear representation is called the {\it left regular
representation} of $G$ on $\gyrL{G}$.
\end{defn}

Let $G$ be a {\it finite} gyrogroup and define a map $\sigma$ by
\begin{equation}\label{eqn: linear functional on Lgyr(G)}
\sigma(f) = \lsum{a\in G}{}f(a),\qquad f\in L(G).
\end{equation}

\begin{thm}\label{thm: property of kernel sigma and Lgyr(G)}
Let $\sigma$ be the map defined by \eqref{eqn: linear functional on Lgyr(G)}. Then the following assertions hold:
\begin{enumerate}
\item\label{item: sigma linear functional}  $\sigma$ is a linear functional.
\item\label{item: dimension of ker sigma} $\dim{(\ker{\sigma})} = \abs{G}-1$.
\item\label{item: invariant subspace of sigma} $\gyrL{G}\cap\ker{\sigma}$ is an invariant subspace of $\gyrL{G}$ under the left regular action of $G$.
\item\label{item: dimension of gyrl(G) cap kernel sigma} either $\gyrL{G}\subseteq\ker{\sigma}$ or $\dim{(\gyrL{G}\cap\ker{\sigma})} = \dim{(\gyrL{G})} - 1$.
\end{enumerate}
\end{thm}
\begin{proof}
The proof of \eqref{item: sigma linear functional} is immediate. Since $\sigma$ is a nonzero linear functional \red{and} $\delta_e\not\in\ker{\sigma}$, it follows from the rank-nullity theorem that 
$$
\dim{(\ker{\sigma})} = \dim{(L(G))} - \dim{(\im{\sigma})} = \abs{G} - 1.
$$
This proves \eqref{item: dimension of ker sigma}.

\eqref{item: invariant subspace of sigma} Set $U = \gyrL{G}\cap\ker{\sigma}$. It is clear that
$$U = \Cset{f\in\gyrL{G}}{\lsum{a\in G}{}f(a) = 0}$$
and that $U$ forms a subspace of $\gyrL{G}$, being the intersection of subspaces. Next, we prove that $U$ is invariant under the action given by \eqref{eqn: left regular action on Lgyr(G)}. Let $g\in G, f\in U$. Then $g\cdot f\in\gyrL{G}$. We compute
\begin{equation}\tag{$\star$}\label{eqn: in proof invariant of kernel sigma}
\lsum{a\in G}{}(g\cdot f)(a) = \lsum{a\in G}{}(f\circ L_{\ominus g})(a) = \lsum{a\in G}{}f(L_{\ominus g}(a)).
\end{equation}
As $a$ runs over all elements of $G$ so does $L_{\ominus g}(a)$ for $L_{\ominus g}$ is a permutation of $G$. From this and \eqref{eqn: in proof invariant of kernel sigma}, we have
$$
\lsum{a\in G}{}(g\cdot f)(a) = \lsum{c\in G}{}f(c) = 0
$$
and so $g\cdot f\in\ker{\sigma}$. It follows that $g\cdot f\in U$, which proves that $U$ is invariant.

\eqref{item: dimension of gyrl(G) cap kernel sigma} Let $\vphi$ be the restriction of $\sigma$ on $\gyrL{G}$. Then $\vphi$ is a linear functional on $\gyrL{G}$. Further, $\ker{\vphi} = \gyrL{G}\cap\ker{\sigma}$.  As $\dim{\F} = 1$, $\dim{(\im{\vphi})} = 0$ or $\dim{(\im{\vphi})} = 1$. By the rank-nullity theorem,  $$\dim{(\ker{\vphi})} = \dim{(\gyrL{G})}$$ or $\dim{(\ker{\vphi})} = \dim{(\gyrL{G})} - 1$. If $\dim{(\ker{\vphi})} = \dim{(\gyrL{G})} - 1$, then we are done. We therefore assume that $\dim{(\ker{\vphi})} = \dim{(\gyrL{G})}$. In this case, we obtain 
$$
\gyrL{G}\cap\ker{\sigma} = \ker{\vphi} = \gyrL{G}
$$ 
and so $\gyrL{G} \subseteq \ker{\sigma}$. 
\end{proof}

Recall that the {\it fixed subspace} of $\gyrL{G}$ is defined as
\begin{equation}
\Fix{\gyrL{G}} = \cset{f\in\gyrL{G}}{a\cdot f = f\textrm{ for all }a\in G}.
\end{equation}
For each $\alpha\in\F$, define a map $f_{\alpha}$ by
\begin{equation}
f_{\alpha}(x) = \alpha,\qquad x\in G.
\end{equation}
It is clear that $f_{\alpha}\in\gyrL{G}$ for all $\alpha\in\F$.

\begin{thm}\label{thm: characterization of Fixed subspace}
Let $G$ be a gyrogroup. Then $\Fix{\gyrL{G}}$ is an invariant subspace of $\gyrL{G}$ and
$$
\Fix{\gyrL{G}} = \cset{f_\alpha}{\alpha\in\F}.
$$
Furthermore, $\set{f_1}$ is a basis for $\Fix{\gyrL{G}}$ so that $\Fix{\gyrL{G}}$ is one-dimensional.
\end{thm}
\begin{proof}
It is routine to check that $\Fix{\gyrL{G}}$ is a subspace of $\gyrL{G}$. Since $a\cdot f = f$ for all $a\in G, f\in\Fix{\gyrL{G}}$, it follows that $\Fix{\gyrL{G}}$ is \mbox{invariant}. Set $W =  \cset{f_\alpha}{\alpha\in\F}$. Let $f\in\Fix{\gyrL{G}}$. For all $a\in G$,
$$
f(a) = (a\cdot f)(a) = (f\circ L_{\ominus a})(a) = f(L_{\ominus a}(a)) = f(\ominus a\oplus a) = f(e).
$$
Thus, $f = f_{f(e)}$ and so  $f\in W$. This proves $\Fix{\gyrL{G}}\subseteq W$. Let $f\in W$. Then $f = f_\alpha$ for some $\alpha\in\F$. Let $a\in G$. Note that $$(a\cdot f)(x) = f(L_{\ominus a}(x)) = \alpha = f(x)$$ for all $x\in G$. Thus, $a\cdot f = f$. Since $a$ is arbitrary, $f\in\Fix{\gyrL{G}}$. This proves $W\subseteq \Fix{\gyrL{G}}$ and so equality holds.

It is clear that $\set{f_1}$ is linearly independent. For each $\alpha\in\F$, note that $$(\alpha f_1)(x) = \alpha 1 = \alpha = f_\alpha(x)$$ for all $x\in G$. Hence, $f_\alpha = \alpha f_1$ and so $\set{f_1}$ spans $\Fix{\gyrL{G}}$.
\end{proof}

\section{The converse of Maschke's theorem}\label{sec: converse of Maschke's theorem}
Throughout this section, we assume that the characteristic of $\F$ is non-zero and $G$ is a finite gyrogroup such that $\cha{\F}$ divides $\abs{G}$. In particular,
$$\underbrace{1+1+\cdots + 1}_{\abs{G}\textrm{ copies}} = 0$$
in $\F$. Under these assumptions on $\F$ and $G$, it follows that
$$
\lsum{a\in G}{}f_\alpha(a) = \lsum{a\in G}{}\alpha = \abs{G}\alpha = 0
$$
and so $f_\alpha\in\ker{\sigma}$, where $\sigma$ is defined by \eqref{eqn: linear functional on Lgyr(G)}, for all $\alpha\in\F$. By what we have proved, the following chain of inclusions between subspaces of $L(G)$ holds:
\begin{equation}\label{eqn: chain of invariant subspace of Lgyr(G)}
\set{0} \subset \Fix{\gyrL{G}} \subseteq \ker{\sigma}\cap\gyrL{G} \subseteq \gyrL{G} \subseteq L(G).
\end{equation}
We remark that the first inclusion is proper since $\dim{(\Fix{\gyrL{G}})} = 1$ (cf. Theorem \ref{thm: characterization of Fixed subspace}) and that the last inclusion is proper if $G$ is not a group (cf. Proposition \ref{prop: relation between Lgyr(G) and L(G)}).

\begin{thm}[Converse of Maschke's theorem for gyrogroups]\label{thm: converse of Maschke's theorem for gyrogroups}
Let $G$ be a finite gyrogroup and let $\F$ be a field such that $\cha{\F}\ne 0$ and $\cha{\F}$ divides $\abs{G}$. If $\sigma(f)\ne 0$ for some $f$ in $\gyrL{G}$, where $\sigma$ is defined by \eqref{eqn: linear functional on Lgyr(G)}, then $\gyrL{G}\cap\ker{\sigma}$ does not possess an invariant direct sum complement in $\gyrL{G}$.
\end{thm}
\begin{proof}
Set $U = \gyrL{G}\cap\ker{\sigma}$. Then \eqref{eqn: chain of invariant subspace of Lgyr(G)} becomes
$$
\set{0} \subset \Fix{\gyrL{G}} \subseteq U \subseteq \gyrL{G} \subseteq L(G).
$$
Assume that $\sigma(f)\ne 0$ for some $f\in\gyrL{G}$. Hence, $\gyrL{G}\not\subseteq \ker{\sigma}$ and so by Theorem \ref{thm: property of kernel sigma and Lgyr(G)} \eqref{item: dimension of gyrl(G) cap kernel sigma}, $\dim{U} = \dim{(\gyrL{G})} - 1$. Assume to the contrary that $U$ has an invariant direct sum complement; that is, there exists an invariant subspace $W$ of $\gyrL{G}$ for which $\gyrL{G} = U\oplus W$. Thus, $\dim{W} = 1$ and so $W = \gen{b}$ for some nonzero vector $b\in\gyrL{G}$. If $g\cdot b = b$ for all $g\in G$, then we would have $b\in\Fix{\gyrL{G}}$ and would have $b\in U$, a contradiction. Hence, $h\cdot b \ne b$ for some $h\in G$. Since $W$ is invariant, it follows that $h\cdot b\in W$ and so $h\cdot b = \lambda_0 b$ for some $\lambda_0\in\F\setminus\set{1}$. Hence, $\sigma(h\cdot b) = \sigma(\lambda_0 b)$. Note that 
$$
\sigma(h\cdot b) = \lsum{a\in G}{}(h\cdot b)(a) = \lsum{a\in G}{}(b\circ L_{\ominus h})(a) = \lsum{a\in G}{}b(L_{\ominus h}(a)) = \lsum{c\in G}{}b(c) 
$$
and that $\sigma(\lambda_0 b) = \lambda_0\sigma(b) = \lambda_0\lsum{a\in G}{}b(a)$. Thus, $\lsum{a\in G}{}b(a) = \lambda_0\lsum{a\in G}{}b(a)$, which implies $\lsum{a\in G}{}b(a) = 0$ because $\lambda_0\ne 1$. This proves $b\in \ker{\sigma}$ and so $b\in U$, a contradiction. Therefore, $W$ does not exist.
\end{proof}

Although the converse to Maschke's theorem for groups is well known, we give an alternative proof using a gyrogroup-theoretic approach.

\begin{thm}[Converse of Maschke's theorem for groups]
If $G$ is a finite group and $\F$ is a field such that $\cha{\F}\ne 0$ and $\cha{\F}$ divides $\abs{G}$, then $\ker{\sigma}$, where $\sigma$ is defined by \eqref{eqn: linear functional on Lgyr(G)}, does not possess an invariant direct sum complement in $L(G)$.
\end{thm}
\begin{proof}
As $G$ is a group, $\gyrL{G} = L(G)$ by Proposition  \ref{prop: relation between Lgyr(G) and L(G)}. Since $\delta_e$, $e$ being the identity of $G$, belongs to $L(G) = \gyrL{G}$ and $\sigma(\delta_e)\ne 0$, it follows from Theorem \ref{thm: converse of Maschke's theorem for gyrogroups} that there is no invariant subspace $W$ of $L(G)$ such that $L(G) = \ker{\sigma}\oplus W$.
\end{proof}

\vspace{0.3cm}
\noindent{\bf Acknowledgements.} The author would like to thank Keng Wiboonton for his collaboration. \red{He is grateful to the referee for careful reading of the  manuscript and invaluable comments.} This research was supported by the Research Fund for DPST Graduate with First Placement, Year 2016, under grant No. 028/2559 and Chiang Mai University.

\bibliographystyle{amsplain}\addcontentsline{toc}{section}{References}
\bibliography{References}
\end{document}